\documentclass[11pt,a4paper,reqno]{amsart}
\usepackage[english]{babel}
\usepackage[applemac]{inputenc}
\usepackage[T1]{fontenc}
\usepackage{palatino}
\usepackage{verbatim}
\usepackage{amsmath}
\usepackage{amssymb}
\usepackage{amsthm}
\usepackage{amsfonts}
\usepackage{graphicx}
\usepackage{mathtools}
\usepackage{enumitem}

\usepackage[colorlinks = true, citecolor = black]{hyperref}
\author{Tuomas Orponen}
\title{On the absolute continuity of radial projections}
\keywords{Radial projections, fractal measures, absolute continuity}
\address{University of Helsinki, Department of Mathematics and Statistics}
\subjclass[2010]{28A80}
\thanks{T.O. is supported by the Academy of Finland via the project \emph{Quantitative rectifiability in Euclidean and non-Euclidean spaces}, grant No. 314172.}
\email{tuomas.orponen@helsinki.fi}

\newcommand{\R}{\mathbb{R}}

\newcommand{\N}{\mathbb{N}}

\newcommand{\calH}{\mathcal{H}}

\newcommand{\calS}{\mathcal{S}}

\newcommand{\spt}{\operatorname{spt}}
\newcommand{\Hd}{\dim_{\mathrm{H}}}

\newcommand{\calM}{\mathcal{M}}

\numberwithin{equation}{section}

\theoremstyle{plain}
\newtheorem{thm}[equation]{Theorem}

\newtheorem{lemma}[equation]{Lemma}

\theoremstyle{definition}

\theoremstyle{remark}
\newtheorem{remark}[equation]{Remark}

\newcommand{\nref}[1]{(\hyperref[#1]{#1})}

\begin{document}

\begin{abstract} Let $d \geq 2$ and $d - 1 < s < d$. Let $\mu$ be a compactly supported Radon measure in $\R^{d}$ with finite $s$-energy. I prove that the radial projections $\pi_{x\sharp}\mu$ of $\mu$ are absolutely continuous with respect to $\calH^{d - 1}$ for every centre $x \in \R^{d} \setminus \spt \mu$, outside an exceptional set of dimension at most $2(d - 1) - s$. This is sharp. In fact, for $x$ outside an exceptional set as above, $\pi_{x\sharp}\mu \in L^{p}(S^{d - 1})$ for some $p > 1$.  \end{abstract}

\maketitle

\section{Introduction}

The space of compactly supported Radon measures on $\R^{d}$ is denoted by $\calM(\R^{d})$. For $x \in \R^{d}$, denote by $\pi_{x} \colon \R^{d} \setminus \{x\} \to S^{d - 1}$ the radial projection
\begin{displaymath} \pi_{x}(y) = \frac{y - x}{|y - x|}, \qquad y \in \R^{d} \setminus \{x\}. \end{displaymath} 
This note is concerned with the question: if $\mu \in \calM(\R^{d})$ has finite $s$-energy for some $d - 1 < s < d$, then how often is $\pi_{x\sharp}\mu$ absolutely continuous with respect to $\calH^{d - 1}|_{S^{d - 1}}$? Write
\begin{displaymath} \calS(\mu) := \{x \in \R^{d} \setminus \spt \mu : \pi_{x\sharp}\mu \text{ is not absolutely continuous w.r.t. } \calH^{d - 1}|_{S^{d - 1}}\}. \end{displaymath}
Note that whenever $x \in \R^{d} \setminus \spt \mu$, the projection $\pi_{x}$ is continuous on $\spt \mu$, and $\pi_{x\sharp}\mu$ is well-defined. One can check that the family of projections $\{\pi_{x}\}_{x \in \R^{d} \setminus \spt \mu}$ fits in the \emph{generalised projections} framework of Peres and Schlag \cite{PS}, and indeed Theorem 7.3 in \cite{PS} yields the estimate
\begin{displaymath} \Hd \calS(\mu) \leq 2d - 1 - s. \end{displaymath}
Combining this bound with standard arguments shows that if $K \subset \R^{d}$ is a Borel set with $d - 1 < \Hd K \leq d$, then 
\begin{equation}\label{PS} \Hd \{x \in \R^{d} : \calH^{d - 1}(\pi_{x}(K)) = 0\} \leq 2d - 1 - \Hd K. \end{equation}
In a fairly recent paper \cite{O}, which built heavily on slightly earlier collaboration \cite{MO} with P. Mattila, I showed that the bound \eqref{PS} is \textbf{not} sharp, and in fact
\begin{equation}\label{visibility} \Hd \{x \in \R^{d} : \calH^{d - 1}(\pi_{x}(K)) = 0\} \leq 2(d - 1) - \Hd K. \end{equation}
The bound \eqref{visibility} is best possible. The proofs in \cite{MO} and \cite{O} were somewhat indirect, and did not improve on the Peres-Schlag bound \eqref{PS} for $\Hd \calS(\mu)$. This improvement is the content of the present note:

\begin{thm}\label{main} If $\mu \in \calM(\R^{d})$ and $I_{s}(\mu) < \infty$ for some $s > d - 1$, then $\Hd \calS(\mu) \leq 2(d - 1) - s$.
\end{thm}
In fact, Theorem \ref{main} follows immediately from the next statement about $L^{p}$-densities:

\begin{thm}\label{mainTechnical} Let $\mu \in \calM(\R^{d})$ as in Theorem \ref{main}. For $p > 1$, write 
\begin{displaymath} \calS_{p}(\mu) := \{x \in \R^{d} \setminus \spt \mu : \pi_{x\sharp}\mu \notin L^{p}(S^{d - 1})\}. \end{displaymath}
Then $\Hd \calS_{p}(\mu) \leq 2(d - 1) - s + \delta(p)$, where $\delta(p) \to 0$ as $p \searrow 1$.
 \end{thm}
 
\begin{remark} Theorem \ref{mainTechnical} can be viewed as an extension of Falconer's exceptional set estimate \cite{Fa} from 1982. I only discuss the planar case. Falconer proved that if $I_{s}(\mu) < \infty$ for some $1 < s < 2$, then the orthogonal projections of $\mu$ to all $1$-dimensional subspaces are in $L^{2}$, outside an exceptional set of dimension at most $2 - s$. Now, orthogonal projections can be viewed as radial projections from points on the line at infinity. Alternatively, if the reader prefers a more rigorous statement, Falconer's proof shows that if $\ell \subset \R^{2}$ is any fixed line outside the support of $\mu$, then all the radial projections of $\mu$ to points on $\ell$ are in $L^{2}$, outside an exceptional set of dimension at most $2 - s$. In comparison, Theorem \ref{mainTechnical} states that the radial projections of $\mu$ to points in $\R^{2} \setminus \spt \mu$ are in $L^{p}$ for some $p > 1$, outside an exceptional set of dimension at most $2 - s$. So, the size of the exceptional set remains the same even if the "fixed line $\ell$" is removed from the statement. The price to pay is that the projections only belong to some $L^{p}$ with $p > 1$ (possibly) smaller than $2$. I do not know, if the reduction in $p$ is necessary, or an artefact of the proof. \end{remark} 
 
 The proof of Theorem \ref{mainTechnical} uses ideas from \cite{MO} and \cite{O}, but is more direct than those arguments, and perhaps a little simpler.
 
 \subsection{Acknowledgements} It seems quite natural to consider the problem of improving the Peres-Schlag estimate for $\Hd \calS(\mu)$, but the question did not occur to me at the time of writing the papers \cite{MO} and \cite{O}. Thanks to Pablo Shmerkin for asking it explicitly, and for nice discussions at Institut Mittag-Leffler in September 2017.
 
 The paper was written during the research programme \emph{Fractal Geometry and Dynamics} at Institut Mittag-Leffler, in fall 2017. I wish to thank both the organisers of the programme, and the staff at the institute, for a pleasant stay.
 
\section{Proof of the main theorem} Fix $\mu \in \calM(\R^{d})$ and $x \in \R^{d} \setminus \spt \mu$. For a suitable constant $c_{d} > 0$ to be determined shortly, consider the weighted measure
\begin{displaymath} \mu_{x} := c_{d}k_{x} \, d\mu, \end{displaymath}
where $k_{x} := |x - y|^{1 - d}$ is the $(d - 1)$-dimensional Riesz kernel, translated by $x$. A main ingredient in the proofs of Theorems \ref{main} and \ref{mainTechnical} is the following identity:

\begin{lemma}\label{lem1} Let $\mu \in C_{0}(\R^{d})$ (that is, $\mu$ is a continuous function with compact support) and $\nu \in \calM(\R^{d})$. Assume that $\spt \mu \cap \spt \nu = \emptyset$. Then, for $p \in (0,\infty)$,
\begin{displaymath} \int \|\pi_{x\sharp} \mu_{x}\|_{L^{p}(S^{d - 1})}^{p} \, d\nu(x) = \int_{S^{d - 1}} \|\pi_{e\sharp} \mu\|_{L^{p}(\pi_{e\sharp}\nu)}^{p} \, d\calH^{d - 1}(e). \end{displaymath}
Here, and for the rest of the paper, $\pi_{e}$ stands for the orthogonal projection onto $e^{\perp} \in G(d,d - 1)$.
\end{lemma}

\begin{proof} Start by assuming that also $\nu \in C_{0}(\R^{d})$. Fix $x \in \R^{d}$. The first aim is to find an explicit expression for the density $\pi_{x} \mu_{x}$ on $S^{d - 1}$, so fix $f \in C(S^{d - 1})$ and compute as follows, using the definition of the measure $\mu_{x}$, integration in polar coordinates, and choosing the constant $c_{d} > 0$ appropriately:
\begin{align*} \int f(e) \, d[\pi_{x\sharp}\mu_{x}](e) & = \int f(\pi_{x}(y)) \, d\mu_{x}(y) = c_{d} \int \frac{f(\pi_{x}(y))}{|x - y|^{d - 1}} \, d\mu(y)\\
& = \int_{S^{d - 1}} f(e) \int_{\R} \mu(x + re) \, dr \, d\calH^{d - 1}(e)\\
& = \int_{S^{d - 1}} f(e) \cdot \pi_{e\sharp}\mu(\pi_{e}(x)) \, d\calH^{d - 1}(e).  \end{align*} 
Since the equation above holds for all $f \in C(S^{d - 1})$, we infer that
\begin{equation}\label{form1} \pi_{x\sharp}\mu_{x} = [e \mapsto \pi_{e\sharp}\mu(\pi_{e}(x))] \, d\calH^{d - 1}|_{S^{d - 1}}. \end{equation}
Now, we may prove the lemma by a straightforward computation, starting with
\begin{align*} \int \|\pi_{x\sharp} \mu_{x}\|_{L^{p}(S^{d - 1})}^{p} \, d\nu(x) & = \int \int_{S^{d - 1}} [\pi_{x\sharp}\mu_{x}(e)]^{p} \, d\calH^{d - 1}(e) \, d\nu(x)\\
& = \int_{S^{d - 1}} \int_{e^{\perp}} \int_{\pi_{e}^{-1}\{w\}} \left[ \pi_{e\sharp}\mu(\pi_{e}(x)) \right]^{p} \nu(x) \, d\calH^{1}(x) \, d\calH^{d - 1}(w) \, d\calH^{d - 1}(e). \end{align*} 
Note that whenever $x \in \pi_{e}^{-1}\{w\}$, then $\pi_{e}(x) = w$, so the expression $[\ldots]^{p}$ above is independent of $x$. Hence,
\begin{align*} \int \|\pi_{x\sharp} \mu_{x}\|_{L^{p}(S^{d - 1})}^{p} \, d\nu(x) & = \int_{S^{d - 1}} \int_{e^{\perp}} \left[\pi_{e\sharp}\mu(t)\right]^{p} \left( \int_{\pi_{e}^{-1}\{w\}} \nu(x) \, d\calH^{1}(x) \right) d\calH^{d - 1}(w) \, d\calH^{1}(e)\\
& = \int_{S^{d - 1}} \int_{e^{\perp}} \left[\pi_{e\sharp}\mu(w)\right]^{p} \pi_{e\sharp}\nu(w) \, d\calH^{d - 1}(w) \, d\calH^{d - 1}(e)\\
& = \int_{S^{d - 1}} \|\pi_{e\sharp} \mu\|_{L^{p}(\pi_{e\sharp}\nu)}^{p} \, d\calH^{d - 1}(e), \end{align*} 
as claimed.

Finally, if $\nu \in \calM(\R^{d})$ is arbitrary, not necessarily smooth, note that 
\begin{displaymath} x \mapsto \|\pi_{x\sharp}\mu_{x}\|_{L^{p}(S^{d - 1})}^{p} \end{displaymath}
is continuous, assuming that $\mu \in C_{0}(\R^{d})$, as we do (to check the details, it is helpful to infer from \eqref{form1} that $\pi_{x}\mu_{x} \in L^{\infty}(S^{d - 1})$ uniformly in $x$, since the projections $\pi_{e\sharp}\mu$ clearly have bounded density, uniformly in $e \in S^{d - 1}$). Thus, if $(\psi_{n})_{n \in \N}$ is a standard approximate identity on $\R^{d}$, we have
\begin{equation}\label{form2} \int \|\pi_{x\sharp}\mu_{x}\|_{L^{p}(S^{d - 1})}^{p} \, d\nu(x) = \lim_{n \to \infty} \int_{S^{d - 1}} \|\pi_{e\sharp} \mu\|_{L^{p}(\pi_{e\sharp}\nu_{n})}^{p} \, d\calH^{d - 1}(e), \end{equation}
with $\nu_{n} = \nu \ast \psi_{n}$. Since $\pi_{e\sharp}\nu_{n}$ converges weakly to $\pi_{e\sharp}\nu$ for any fixed $e \in S^{d - 1}$, and $\pi_{e\sharp}\mu \in C_{0}(e^{\perp})$, it is easy to see that the right hand side of \eqref{form2} equals
\begin{displaymath} \int_{S^{d - 1}} \|\pi_{e\sharp}\mu\|_{L^{p}(\pi_{e\sharp}\nu)}^{p} \, d\calH^{d - 1}(e). \end{displaymath}
This completes the proof of the lemma. \end{proof}

We can now prove Theorem \ref{mainTechnical}, which implies Theorem \ref{main}.

\begin{proof}[Proof of Theorem \ref{main}] Fix $2(d - 1) - s < t < d - 1$. It suffices to prove that if $\nu \in \calM(\R^{d})$ is a fixed measure with $I_{t}(\nu) < \infty$, and $\spt \mu \cap \spt \nu = \emptyset$, then
\begin{displaymath} \pi_{x\sharp}\mu_{x} \in L^{p}(S^{d - 1}) \qquad \text{for } \nu \text{ a.e. } x \in \R^{d}, \end{displaymath}
whenever
\begin{equation}\label{assumptions} 1 < p \leq \min\left\{2 - \tfrac{t}{(d - 1)},\tfrac{t}{2(d - 1) - s}\right\}. \end{equation}
We will treat the numbers $d,p,s,t$ as "fixed" from now on, and in particular the implicit constants in the $\lesssim$ notation may depend on $d,p,s,t$. Note that the right hand side of \eqref{assumptions} lies in $(1,2)$, so this is a non-trivial range of $p$'s. Fix $p$ as in \eqref{assumptions}. The plan is to show that
\begin{equation}\label{form4} \int \|\pi_{x\sharp}\mu_{x}\|_{L^{p}}^{p} \, d\nu(x) < \infty. \end{equation}  
This will be done via Lemma \ref{lem1}, but we first need to reduce to the case $\mu \in C_{0}(\R^{d})$. Let $(\psi_{n})_{n \in \N}$ be a standard approximate identity on $\R^{d}$, and write $\mu_{n} = \mu \ast \psi_{n}$. Then $\pi_{x\sharp}(\mu_{n})_{x}$ converges weakly to $\pi_{x\sharp}\mu_{x}$ for any fixed $x \in \spt \nu \subset \R^{d} \setminus \spt \mu$:
\begin{displaymath} \int f(e) \, d[\pi_{x\sharp}\mu_{x}(e)] = \lim_{n \to \infty} \int f(e) \, d\pi_{x\sharp}(\mu_{n})_{x}(e), \qquad f \in C(S^{d - 1}). \end{displaymath}
It follows that
\begin{displaymath} \|\pi_{x\sharp}\mu_{x}\|_{L^{p}(S^{d - 1})}^{p} \leq \liminf_{n \to \infty} \|\pi_{x\sharp}(\mu_{n})_{x}\|_{L^{p}(S^{d - 1})}^{p}, \quad x \in \spt \nu, \end{displaymath}
and consequently
\begin{displaymath} \int \|\pi_{x\sharp}\mu_{x}\|_{L^{p}(S^{d - 1})}^{p} \, d\nu(x) \leq \liminf_{n \to \infty} \int \|\pi_{x\sharp}(\mu_{n})_{x}\|_{L^{p}(S^{d - 1})}^{p} \, d\nu(x) \end{displaymath}
by Fatou's lemma. Now, it remains to find a uniform upper bound for the terms on the right hand side; the only information about $\mu_{n}$, which we will use, is that $I_{s}(\mu_{n}) \lesssim I_{s}(\mu)$. With this in mind, we simplify notation by denoting $\mu_{n} := \mu$. For the remainder of the proof, one should keep in mind that $\pi_{e\sharp}\mu \in C_{0}^{\infty}(e^{\perp})$ for $e \in S^{d - 1}$, so the integral of $\pi_{e\sharp}\mu$ with respect to various Radon measures on $e^{\perp}$ is well-defined, and the Fourier transform of $\pi_{e\sharp}\mu$ on $e^{\perp}$ (identified with $\R^{d - 1}$) is a rapidly decreasing function. 

We start by appealing to Lemma \ref{lem1}:
\begin{equation}\label{form3} \int \|\pi_{x\sharp}\mu_{x}\|_{L^{p}(S^{d - 1})}^{p} \, d\nu(x) = \int_{S^{d - 1}} \|\pi_{e\sharp} \mu\|_{L^{p}(\pi_{e\sharp}\nu)}^{p} \, d\calH^{d - 1}(e). \end{equation}
Next, we estimate the $L^{p}(\pi_{e\sharp}\nu)$-norms of $\pi_{e\sharp}\mu$ individually, for $e \in S^{d - 1}$ fixed. We start by recording the standard fact that $I_{t}(\pi_{e\sharp}\nu) < \infty$ for $\calH^{d - 1}$ almost every $e \in S^{d - 1}$, and we will only consider those $e \in S^{d - 1}$ satisfying this condition. Recall that $1 < p \leq t/[2(d - 1) - s]$. Fix $f \in L^{q}(\pi_{e\sharp}\nu)$, with $q = p'$ and $\|f\|_{L^{q}(\pi_{e\sharp}\nu)} = 1$, and note that
\begin{displaymath} I_{2(d - 1) - s}(f \, d\pi_{e\sharp}\nu) = \iint \frac{f(x)f(y) \, d\pi_{e\sharp}\nu(x) \, d\pi_{e\sharp}\nu(y)}{|x - y|^{2(d - 1) - s}} \lesssim I_{t}(\pi_{e\sharp}\nu)^{1/p} \end{displaymath}
by H\"older's inequality. It now follows from Theorem 17.3 in \cite{Mat} that
\begin{align*} \int \pi_{e\sharp}\mu \cdot f \, d\pi_{e\sharp}\nu & \lesssim \sqrt{I_{2(d - 1) - s}(f \, d\pi_{e\sharp}\nu)}\|\pi_{e\sharp}\mu\|_{H^{[s - (d - 1)]/2}}\\
& \lesssim \left(I_{t}(\pi_{e\sharp}\nu) \right)^{1/2p}\left(\int_{e^{\perp}} |\widehat{\pi_{e\sharp}\mu}(\xi)|^{2}|\xi|^{s - (d - 1)} \, d\xi \right)^{1/2}. \end{align*}
Since the function $f \in L^{q}(\pi_{e\sharp}\nu)$ with $\|f\|_{L^{q}(\pi_{e\sharp}\nu)} = 1$ was arbitrary, we may infer by duality that
\begin{displaymath} \|\pi_{e\sharp}\mu\|_{L^{p}(\pi_{e\sharp}\nu)} \lesssim \left(I_{t}(\pi_{e\sharp}\nu)\right)^{1/2p}\left(\int_{e^{\perp}} |\widehat{\pi_{e\sharp}\mu}(\xi)|^{2}|\xi|^{s - (d - 1)} \, d\xi \right)^{1/2}. \end{displaymath}
We can finally estimate \eqref{form3}. We use duality once more, so fix $f \in L^{q}(S^{d - 1})$ with $\|f\|_{L^{q}(S^{d - 1})} = 1$. Then, write
\begin{align*} \int_{S^{d - 1}} & \|\pi_{e\sharp}\mu\|_{L^{p}(\pi_{e\sharp}\nu)} \, \cdot f(e) \, d\calH^{d - 1}(e)\\
& \lesssim \int_{S^{d - 1}} \left(I_{t}(\pi_{e\sharp}\nu)\right)^{1/2p}\left(\int_{e^{\perp}} |\widehat{\pi_{e\sharp}\mu}(\xi)|^{2}|\xi|^{s - (d - 1)} \, d\xi \right)^{1/2} \cdot f(e) \, d\calH^{d - 1}(e)\\
& \lesssim \left( \int_{S^{d - 1}} I_{t}(\pi_{e\sharp}\nu)^{1/p} \cdot f(e)^{2} \, d\calH^{d - 1}(e) \right)^{1/2}\left(\int_{S^{d - 1}} \int_{e^{\perp}} |\widehat{\pi_{e\sharp}\mu}(\xi)|^{2}|\xi|^{s - (d - 1)} \, d\xi \, d\calH^{d - 1}(e) \right)^{1/2}. \end{align*} 
The second factor is bounded by $\lesssim I_{s}(\mu)^{1/2} < \infty$, using (generalised) integration in polar coordinates, see for instance (2.6) in \cite{MO}. To tackle the first factor, say "$I$", write $f^{2} = f \cdot f$ and use H\"older's inequality again:
\begin{displaymath} I \lesssim \left(\int_{S^{d - 1}} I_{t}(\pi_{e\sharp}\nu) \, \cdot f(e)^{p} \, d\calH^{d - 1}(e) \right)^{1/2p} \cdot \|f\|_{L^{q}(S^{d - 1})}^{1/2} \end{displaymath}
The second factor equals $1$. To see that the first factor is also bounded, note that if $B(e,r) \subset S^{d - 1}$ is a ball, then
\begin{displaymath} \int_{B(e,r)} f^{p} \, d\calH^{d - 1} \leq \left(\calH^{d - 1}(B(e,r))\right)^{2 - p} \cdot \left( \int_{S^{d - 1}} f^{q} \, d\calH^{d - 1} \right)^{p - 1} \lesssim r^{(d - 1)(2 - p)}. \end{displaymath}
Thus, $\sigma = f^{p} \, d\calH^{d - 1}$ is a Frostman measure on $S^{d - 1}$ with exponent $(d - 1)(2 - p)$. Now, it is well-known (and first observed by Kaufman \cite{Ka}) that
\begin{displaymath} \int_{S^{d - 1}} I_{t}(\pi_{e\sharp}\nu) \, d\sigma(e) = \iint \int_{S^{d - 1}} \frac{d\sigma(e)}{|\pi_{e}(x) - \pi_{e}(y)|^{t}} \, d\nu(x) \, d\nu(y) \lesssim I_{t}(\nu), \end{displaymath}
as long as $t < (d - 1)(2 - p)$, which is implied by \eqref{assumptions}. Hence $I \lesssim I_{t}(\nu)^{1/2p}$, and finally
\begin{displaymath} \int_{S^{d - 1}} \|\pi_{e\sharp}\mu\|_{L^{p}(\pi_{e\sharp}\nu)} \, \cdot f(e) \, d\calH^{d - 1}(e) \lesssim I_{t}(\nu)^{1/2p}I_{s}(\mu)^{1/2} \end{displaymath}
for all $f \in L^{q}(S^{d - 1})$ with $\|f\|_{L^{q}(S^{d - 1})} = 1$. By duality, it follows that
\begin{displaymath} \eqref{form3} \lesssim I_{t}(\nu)^{1/2p}I_{s}(\mu)^{1/2} < \infty. \end{displaymath}
This proves \eqref{form4}, using \eqref{form3}. The proof of Theorem \ref{mainTechnical} is complete. \end{proof}

\end{document}